\documentclass[conference]{IEEEtran}
\usepackage{cite}

\usepackage{amsmath}
\usepackage{amssymb}
\usepackage{commath}
\usepackage{amsthm}
\usepackage{color}
\usepackage{newtxtext}      
\usepackage{newtxmath}

\newtheorem{theorem}{Theorem}
\newtheorem{proposition}[theorem]{Proposition}

\theoremstyle{definition}
\newtheorem{example}[theorem]{Example}
\newtheorem{remark}[theorem]{Remark}
\newtheorem{definition}[theorem]{Definition}

\newcommand{\RR}{\mathbb{R}}

\newcommand{\NN}{\mathbb{N}}
\newcommand{\ZZ}{\mathbb{Z}}
\DeclareMathOperator{\LL}{L}
\DeclareMathOperator*{\esssup}{ess\,sup}

\newcommand{\hilbert}{\mathcal{H}}
\newcommand{\repr}{{\mathcal{R}}}
\newcommand{\uSquare}{\mathbb{K}}
\newcommand{\sqr}{{\uSquare + \PhVarB}}
\DeclareMathOperator{\Id}{Id}
\newcommand{\osc}{\Omega} 
\newcommand{\discStar}{\mathcal{D}^*}
\newcommand{\discTrans}{{\mathcal{D}^*_{\mathrm{shift}}}}
\newcommand{\discDil}{{\mathcal{D}^*_{\mathrm{dil}}}}

\newcommand{\PhVar}{\eta}
\newcommand{\PhVarA}{\nu}
\newcommand{\PhVarB}{{\rho}}

\newcommand{\PhSpace}{X}
\newcommand{\PhSpaceDis}{\Lambda}

\usepackage[pdftex,
colorlinks, 
bookmarksnumbered, 
bookmarksopen, 
linkcolor=blue, 
citecolor=blue, 
]{hyperref}
\def\BibTeX{{\rm B\kern-.05em{\sc i\kern-.025em b}\kern-.08em
    T\kern-.1667em\lower.7ex\hbox{E}\kern-.125emX}}

\IEEEoverridecommandlockouts

\begin{document}
\author{
\IEEEauthorblockN{Jan Zimmermann\IEEEauthorrefmark{1}, Andreas Klotz, and Nicki Holighaus}
\IEEEauthorblockA{Acoustics Research Institute, Austrian Academy of Sciences, Dominikanerbastei 16, 1010 Wien, Austria.}
E-Mail: \emph{\{jan.zimmermann, andreas.klotz, nicki.holighaus\}@oeaw.ac.at}
\IEEEauthorblockA{\IEEEauthorrefmark{1}\emph{Corresponding Author}
}
}
\date{\today}
\title{Discretization of Continuous Frames by Quasi-Monte Carlo Methods
\thanks{This work was supported by the Austrian Science Fund (FWF) [\href{https://doi.org/10.55776/PAT4780023}{10.55776/PAT4780023}]}}

\maketitle

\begin{abstract}
We introduce a discretization scheme for continuous localized frames using quasi–Monte Carlo integration and discrepancy theory. By generalizing classical concepts, we define a discrepancy measure on the entire phase space $\mathbb R^2$
  and establish a corresponding Koksma–Hlawka inequality. This approach enables control over the density of the discretized frame and ensures the universality of the sampling set, relying only on the discrepancy of the sampling set and on the Sobolev-type seminorm of an iterated kernel rather than on specific frame properties.
\end{abstract}


\section{Introduction}

Continuous frames play a fundamental role in harmonic analysis and signal processing. Let \(\mathcal{H}\) be a separable Hilbert space and \((\PhSpace,\mu)\) a measure space. A family
$
\Psi = (\psi_\PhVar)_{\PhVar\in\PhSpace} \subset \mathcal{H}
$
is called a \emph{Parseval frame} if the mapping \(\PhVar\mapsto \psi_\PhVar\) is weakly measurable and
\[
\|f\|_{\mathcal{H}}^2 = \int_{\PhSpace} |\langle f,\psi_\PhVar\rangle|^2 \, d\mu(\PhVar),\quad \text{for all } f\in\mathcal{H}.
\]
In this setting the analysis operator
\[
C_\Psi: \mathcal{H} \to \LL^2(\PhSpace),\quad f\mapsto (\langle f,\psi_\PhVar\rangle)_{\PhVar\in\PhSpace},
\]
acts as an isometry onto the reproducing kernel Hilbert space \(\mathcal{R} := C_\Psi(\mathcal{H})\) with reproducing kernel
\begin{equation}\label{eq:reproducingKernel}
R(\PhVarA,\PhVar)=\langle \psi_\PhVar,\psi_{\PhVarA}\rangle.
\end{equation}

The discretization of continuous frames—that is, replacing the integral in the frame condition by a discrete sum over a set \(\PhSpaceDis\subset \PhSpace\), such that 
\[ A \norm{f}_\hilbert^2 \leq \sum_{\lambda\in\Lambda}  |\langle f, \psi_\lambda  \rangle|^2 \leq B \norm{f}_\hilbert^2 \]
for some constants (the \emph{frame bounds})  $0<A\leq B< \infty$ and all $f\in\hilbert$ ---
has been solved in great generality by Freeman and Speegle \cite{Freeman2019}. However, in their framework, the phase space is only assumed to carry a measure, so no further information on the quality of the discretization beyond the ratio of the frame bound $B/A$ can be expected.

For \emph{localized frames}, the situation is considerably more refined. Here the phase space \(\PhSpace\) is endowed with a metric, and the associated reproducing kernel \(R\) exhibits decay properties. In particular, one typically requires that \(R\) satisfies Schur’s test, i.e. that the Schur norm of $R$, given by
\[
\norm{R}_{\mathcal{A}_1}= \esssup_{\PhVarA\in\PhSpace}  \int_{\PhSpace}|R(\PhVarA,\PhVar)| d\PhVar,
  \]
is finite. Under such localization and for sufficiently ``tame'' geometry of the  phase space, one may introduce a notion of \emph{sampling density}---the average number of samples per unit volume---and even identify a critical density \(D^{-}\) dictated by the geometry of \(\PhSpace\), see~\cite{Fuehr2017} for details. In this context, discretizations can be compared in terms of how closely their sampling densities approach \(D^{-}\) while maintaining a small frame bound ratio.

A promising strategy to achieve favorable sampling sets
is to select sampling points that are as uniformly distributed as possible, relative to the geometry of \(\PhSpace\). In numerical integration, this quality of a point set is measured by its \emph{discrepancy}. In the simplest setting the \emph{quasi--Monte Carlo} (QMC) method approximates an integral of a function $h$ over \([0,1]^d\) by
\[
\int_{[0,1]^d} h(\PhVar)\,d\PhVar \approx \frac{1}{N} \sum_{n=1}^N h(p_n),
\]
for  a set of sample points $\mathcal{P}=\{p_1,\ldots,p_N\}\subset [0,1]^d$. A \emph{Koksma--Hlawka inequality} links the integration error to the \emph{star discrepancy} of the point set $\mathcal{P}$,
\[
  \mathcal{D}^*(\mathcal{P})= \sup_{\PhVar\in [0,1]^d} \left|\frac{\#(\mathcal{P}\cap [0,\PhVar])}{N} - \mu([0,\PhVar])\right| \,,
\]
where $\mu$ is the Lebesgue measure, and a measure of the smoothness of \(h\). For example, in two dimensions one has:

\medskip
\begin{theorem}[{Koksma--Hlawka inequality \cite[Prop. 2.18]{Dick2010}}]\label{thm:koksmaHlawka}
Let \(\mathcal{P}\subset [0,1]^2\) be a set of \(N\) points, and let \(h:[0,1]^2\to\mathbb{C}\) be a function with continuous mixed partial derivatives $\partial_1 h$, $\partial_2 h$ and $\partial_{12}h$. Then the quadrature error
\[
e_0(h,\mathcal{P}) = \int_{[0,1]^2} h(\PhVar)\,d\PhVar - \frac{1}{N} \sum_{p\in\mathcal{P}} h(p)
\]
satisfies
\begin{align*}
  \bigl|e_0(h,\mathcal{P})\bigr| &\leq \mathcal{D}^*(\mathcal{P}) 
  \Biggl[ \int_{[0,1]^2} |\partial_{12}h(\PhVar)|\,d\PhVar \\
   &+ \int_0^1 |\partial_1 h(\PhVar_1,1)|\,d\PhVar_1 + \int_0^1 |\partial_2 h(1,\PhVar_2)|\,d\PhVar_2 \Biggr].
\end{align*}
\end{theorem}

Low discrepancy sets thus yield small quadrature errors. Classical constructions achieve star discrepancies of order \(\mathcal{O}((\log N)^{d-1}/N)\) in dimension \(d\).

In this work we combine ideas from frame theory and quasi--Monte Carlo integration to obtain new discretization results for continuous, localized frames.
Our first key observation is that the Schur norm $\,\norm{\bullet}_{\mathcal{A}_ 1}$ of the difference between the reproducing kernel
and its discretization can be expressed as the supremum of a family of quadrature errors. On compact subsets of the phase space, these errors are controlled by the product of the star discrepancy of the sampling set and a mixed Sobolev norm 
of an associated kernel.

Our contributions are as follows.  First, we extend the Koksma--Hlawka (KH)  inequality to integrals over the entire Euclidean space $\RR^2$ and adapt the definition of star discrepancy to this new setting.   By invoking results of Skriganov~\cite{Skriganov1994}, we  can infer that for certain lattices \(\PhSpaceDis\subset \mathbb{R}^2\), the point set $a\Lambda$ achieves an asymptotic global discrepancy of order
\[ \mathcal{O} \big( a^2 \ln(2+a^{-1})\big) \]
for $a\to 0$, see Section \ref{sec:admissibleLattices}.
This extension represents, to the best of our knowledge, the first instance of QMC-type integration on a space of infinite measure. Second, we apply this framework to derive a new discretization result for continuous frames in suitable localization classes (Theorem \ref{thm:sufficientConditionDiscretization}). In particular, this result implies \emph{universality} of the sampling set: The existence of a frame discretization for a frame $\Psi$ and sampling set $\Lambda$ only depends on its discrepancy and on a Sobolev like seminorm on its reproducing kernel, not on specifities of the kernel itself.

As a proof of concept, we discretize the short-time Fourier transform (STFT) and demonstrate that any lattice that is admissible in the sense of Skriganov~\cite{Skriganov1994} yields discretizations with universal frame bounds invariant under arbitrary dilations of the generating Gabor window.

By bridging continuous frame theory and quasi--Monte Carlo integration, our work provides a systematic approach to constructing discrete frames with controlled density and stability, thereby enhancing both the theoretical understanding and practical implementation of frame discretization.

\section{From Continuous Frames to Quasi-Monte Carlo}\label{sec:framesToQMC}

From now on we assume that the continuous frame $\Psi$ is defined on the Euclidean plane $X=\RR^2$, equipped with the Lebesgue measure $\mu$. The following approach to frame discretization can be found in many places, see for instance~\cite{fornasier2005,Groechenig1991,Dahlke2008}. Let $\Lambda\subset\RR^2$ be a discrete set and $(a_\lambda)_{\lambda\in\Lambda}$ family of positive numbers. We want to find a sufficient condition for the family $(\sqrt{a_\lambda}\psi_\lambda)_{\lambda\in \Lambda}$ to be a discrete frame for $\hilbert$, that is, to satisfy
\[ A \norm{f}_\hilbert^2 \leq \sum_{\lambda\in\Lambda} a_\lambda |\langle f, \psi_\lambda  \rangle|^2 \leq B \norm{f}_\hilbert^2 \]
for some constants $0<A\leq B< \infty$ and all $f\in\hilbert$. The inclusion of the weights $(a_\lambda)_{\lambda\in\Lambda}$ above will become relevant in Section \ref{sec:koksmaHlawka}. 
This is the case 
if and only if the frame operator
\[ S_\Lambda:\hilbert\to\hilbert, \quad S_\Lambda f = \sum_{\lambda\in \Lambda} a_\lambda \langle f, \psi_\lambda\rangle \psi_\lambda \]
is bounded and stably invertible. The optimal frame bounds are then  $A=\|S_\Lambda^{-1}\|_{\hilbert\to\hilbert}^{-1}$ and  $B=\norm{S_\Lambda}_{\hilbert\to\hilbert}$, cf. ~\cite[Prop. 5.4.4]{Christensen2016}.

As the analysis operator $C_\Psi:\hilbert\to\repr$ is an isometry, the push forward of  $S_\Lambda$ by $C_\Psi$, which is given by $T_\Lambda = C_\Psi S_\Lambda C_\Psi^{-1}$, has the same spectral properties as $S_\Lambda$. In particular, we have $\|T_\Lambda^{\pm 1}\|_{\repr\to\repr} = \|S_\Lambda^{\pm 1}\|_{\hilbert\to\hilbert}$. A short computation reveals that $T_\Lambda$ can be written as an integral operator
\[ T_\Lambda F(\PhVar) = \int_{\RR^2} R_\Lambda(\PhVar,\PhVarA) F(\PhVarA) \,d\PhVarA,\quad F\in \repr \]
with the kernel
\[ R_\Lambda(\PhVar,\PhVarA) = \sum_{\lambda\in \Lambda} a_\lambda R(\PhVar,\lambda) R(\lambda,\PhVarA), \quad \PhVar, \PhVarA\in\RR^2\,, \]
$R$ being the reproducing kernel (\ref{eq:reproducingKernel}).

By the standard Neumann series argument, stable invertibility of $T_\Lambda$ follows if   $\norm{\Id_\repr - T_\Lambda}_{\repr\to\repr}<1$.
As 
$  F(\PhVar) = \int_{\RR^2} R(\PhVar,\PhVarA) F(\PhVarA) \,d\PhVarA $
for $F \in \repr$ we have to bound the operator norm of
\[ F \mapsto \int_{\RR^2} \left( R(\,\cdot\,,\PhVarA) - R_\Lambda(\,\cdot\,,\PhVarA) \right) F(\PhVarA) \,d\PhVarA \]
on $\repr$ from above.  Schur's test yields 
\begin{equation}\label{eq:SchurTest}
\| \Id_\repr - T_\Lambda \|_{\repr\to\repr} \leq \esssup_{\PhVarA\in \RR^2} \int_{\RR^2} \left| R(\PhVar,\PhVarA) - R_\Lambda(\PhVar,\PhVarA) \right| \,d\PhVar.
\end{equation}

We note that the reproducing kernel $R(\PhVar,\PhVarA)$ itself satisfies the reproducing formula $R(\PhVar,\PhVarA) = \int_{\RR^2} R(\PhVar,\PhVarB)R(\PhVarB,\PhVarA)\,d\PhVarB$. Thus, for fixed $\PhVar,\PhVarA\in\RR^2$, the integrand on the right hand side of (\ref{eq:SchurTest}) is equal to
\begin{align}\label{eq:quadratureOfReproducingKernel}
\left| \int_{\RR^2} R(\PhVar, \PhVarB) R(\PhVarB,\PhVarA) \,d\PhVarB - \sum_{\lambda\in\Lambda} a_\lambda R(\PhVar,\lambda) R(\lambda,\PhVarA) \right|.
\end{align}

Let us now write the  weighted quadrature error of a continuous function $h\in \LL^1(\RR^2)$ as
\begin{equation}\label{eq:quadratureErrorDef}
e(h,\Lambda) = \int_{\RR^2} h(\PhVarB) \,d\PhVarB - \sum_{\lambda\in\Lambda} a_\lambda h(\lambda).
\end{equation}
 With the notation  $K^{(\PhVar,\PhVarA)} = R(\PhVar, \,\cdot\,) R(\,\cdot\, ,\PhVarA)$, the expression (\ref{eq:quadratureOfReproducingKernel}) can then be identified as the quadrature error  $|e( K^{(\PhVar,\PhVarA)}, \Lambda) |$.
 
\begin{proposition}\label{prop:discretizationSufficientCondition}
If the kernel $K^{(\PhVar,\PhVarA)}$ satisfies the condition
\[ \epsilon = \esssup_{\PhVarA\in\RR^2} \int_{\RR^2} \left| e\big( K^{(\PhVar,\PhVarA)}, \Lambda \big) \right| \,d\PhVar < 1, \]
then $(\sqrt{a_\lambda}\psi_\lambda)_{\lambda\in\Lambda}$ is a discrete frame for $\hilbert$ with the frame bounds $1-\epsilon$ and $1+\epsilon$.
\end{proposition}
\begin{proof}
Resubstituting Equation (\ref{eq:quadratureOfReproducingKernel}) into (\ref{eq:SchurTest}) yields a sufficient condition for the stable invertibility of $T_\Lambda$. Upper and lower frame bounds can be estimated as
 $ \|T_\Lambda\|_{\repr\to\repr}  \leq 1+ \|\Id_\repr - T_\Lambda\|_{\repr\to\repr}$, and  
$ \|T_\Lambda^{-1}\|_{\repr\to\repr}^{-1} \geq 1 - \| \Id_\repr - T_\Lambda \|_{\repr\to\repr}$
\end{proof}

\section{Quadrature Error Estimates for Integrals over $\RR^2$}\label{sec:koksmaHlawka}
In this section, we aim to find an upper bound for the quadrature error as defined in (\ref{eq:quadratureErrorDef}). For that, we prove an analogue to Theorem \ref{thm:koksmaHlawka} for functions on $\RR^2$. However, our approach leads to a \emph{weighted} quadrature rule as in (\ref{eq:quadratureErrorDef}), with weights $(a_\lambda)_{\lambda\in\Lambda}$ that are determined by the local structure of $\Lambda$.

We assume $h\in \LL^1(\RR^2)$ to have continuous mixed partial derivatives $\partial_1 h$, $\partial_2h$ and $\partial_{12}h$. We further assume that each square of side-length $1$ contains at least one point from $\Lambda$. That is, if we write $\uSquare=[0,1]^2$ for the unit square and $\# M$ for the cardinality of some set $M$, we have $N_\PhVarB = \# \Lambda\cap(\sqr) \geq 1$ for all $\PhVarB\in\RR^2$. Note that $N_\PhVarB$ is always finite as $\Lambda$ is discrete.

To find an estimate for $e(h,\Lambda)$, we start with the expression
\begin{equation}\label{eq:partitioning}
\int_{\RR^2} \left( \int_{\sqr} h(\PhVar) \,d\PhVar - \frac{1}{N_\PhVarB} \sum_{\lambda\in \Lambda\cap (\sqr)} h(\lambda) \right) \,d\PhVarB.
\end{equation}
Denoting the characteristic function of the set $\sqr$ by $\chi_{\sqr}$, the integral can immediately be simplified and is equal to
\[ \int_{\RR^2} h(\PhVar) \,d\PhVar - \sum_{\lambda\in\Lambda} h(\lambda) \int_{\RR^2} \frac{\chi_{\sqr}(\lambda)}{N_\PhVarB} \,d\PhVarB. \]
Thus, (\ref{eq:partitioning}) describes the quadrature error $e(h,\Lambda)$ with respect to the weights
\begin{equation}\label{eq:quadratureWeightFormula}
a_\lambda = \int_{\RR^2} \frac{\chi_{\sqr}(\lambda)}{N_\PhVarB} \,d\PhVarB.
\end{equation}

To estimate (\ref{eq:partitioning}) from above, we first take a look at its inner part
\[ e_\PhVarB(h,\Lambda) = \int_{\sqr} h(\PhVar) \,d\PhVar - \frac{1}{N_\PhVarB} \sum_{\lambda\in \Lambda\cap (\sqr)} h(\lambda) \]
for fixed $\PhVarB \in\RR^2$. If we take care of the translation by $\PhVarB$, we can apply Theorem \ref{thm:koksmaHlawka} to this difference. For that, we define the star discrepancy of $\Lambda$ anchored at $\PhVarB$ by
\[ \discStar_\PhVarB(\Lambda) = \sup_{\PhVar\in [0,1]^2} \left| \frac{\# \Lambda\cap \PhVarB+[0,\PhVar]}{N_\PhVarB} - \mu([0,\PhVar]) \right|. \]
Writing $\rho = (\rho_1, \rho_2)$, we obtain the inequality
\begin{multline}\label{eqLocalQuadratureEstimate}
|e_\PhVarB(h,\Lambda)| \leq \discStar_\PhVarB(\Lambda) \Bigg[ \int_0^1 |\partial_1 h(\PhVarB_1+\PhVar_1, \PhVarB_2+1)| \,d\PhVar_1 \\
+ \int_0^1 |\partial_2 h(\PhVarB_1+1, \PhVarB_2+\PhVar_2)| \,d\PhVar_2 + \int_{[0,1]^2} |\partial_{12}h(\PhVarB+\PhVar)| \,d\PhVar \Bigg].
\end{multline}
Now, combining (\ref{eq:partitioning}) with (\ref{eqLocalQuadratureEstimate}) and applying Hölder's inequality yields our global quadrature rule.
\begin{proposition}\label{prop:globalKoksmaHlawka}
Let $\discTrans(\Lambda) = \sup_{\PhVarB\in\RR^2} \discStar_\PhVarB(\Lambda)$. Then we have the inequality
\begin{align*}
|e(h, \Lambda)| &\leq \int_{\RR^2} |e_\PhVarB(h,\Lambda)| \,d\PhVarB \nonumber\\
&\leq \discTrans(\Lambda) \left[ \,\norm{\partial_1 h}_{\LL^1} + \norm{\partial_2 h}_{\LL^1} + \norm{\partial_{12} h}_{\LL^1}\right],
\end{align*}
where the $\LL^1$-norms are taken over all of $\RR^2$.
\end{proposition}

Together, Propositions \ref{prop:discretizationSufficientCondition} and \ref{prop:globalKoksmaHlawka} imply a sufficient condition for $\Lambda$ to induce a discrete frame as described in Section \ref{sec:framesToQMC}.

\begin{theorem}\label{thm:sufficientConditionDiscretization}
Using the notation from Section \ref{sec:framesToQMC}, assume that $K^{(\PhVar,\PhVarA)}$ has continuous mixed partial derivatives for almost all $\PhVar,\PhVarA\in\RR^2$. Let
\begin{multline}\label{eq:oscillationOfFrame}
\osc(\Psi) = \esssup_{\PhVarA\in\RR^2} \int_{\RR^2} \Big[ \,\norm{\partial_1 K^{(\PhVar,\PhVarA)}}_{\LL^1} \\
+ \norm{\partial_2 K^{(\PhVar,\PhVarA)}}_{\LL^1} + \norm{\partial_{12} K^{(\PhVar,\PhVarA)}}_{\LL^1}\Big] \,d\PhVar.
\end{multline}
If $\discTrans(\Lambda) \osc(\Psi) < 1$, then the family $(\sqrt{a_\lambda}\psi_\lambda)_{\lambda\in\Lambda}$ is a discrete frame for $\hilbert$ with the frame bounds $1- \discTrans(\Lambda) \osc(\Psi)$ and $1+\discTrans(\Lambda) \osc(\Psi)$.
\end{theorem}

The discrepancy $\discTrans(\Lambda)$ is a quantitative measure on how uniformly the set $\Lambda$ is spread in $\RR^2$. If we replace $\Lambda$ by $a\Lambda$ with small $a>0$, we can make the discrepancy $\discTrans(a\Lambda)$ as small as we like. However, the structure of $\Lambda$ is of particular importance, as it decides \emph{how fast} $\discTrans(a\Lambda)$ becomes small. A quickly decreasing discrepancy is essential if we want to find discretizations where the set $a\Lambda$ has controllable density.

In Section \ref{sec:admissibleLattices}, we describe certain lattices for which the discrepancy actually has a fast decay, and which can therefore be used effectively in Theorem \ref{thm:sufficientConditionDiscretization}.

\begin{remark}
By assumption, we have $N_\rho\geq 1$ for all $\PhVarB\in\RR^2$, so (\ref{eq:quadratureWeightFormula}) implies that $a_\lambda\leq 1$ for all $\lambda\in\Lambda$. If we further assume $\Lambda$ to be separated, i.e. $N_\PhVarB\leq n$ for some $n\in\NN$, it also follows that $a_\lambda \geq 1/n$ for all $\lambda\in\Lambda$. Thus, the family $(\sqrt{a_\lambda} \psi_\lambda)_{\lambda\in\Lambda}$ is a frame for $\hilbert$ if and only if $(\psi_\lambda)_{\lambda\in\Lambda}$ is one, allowing us to remove the weights from the frame in Theorem \ref{thm:sufficientConditionDiscretization}. Of course, this will impact the frame bounds.
\end{remark}

\begin{remark}\label{rem:weightsForLattices}
If $\Lambda$ is a lattice, we have $a_\lambda = \det(\Lambda)$ for all $\lambda\in\Lambda$, where the determinant of the lattice $\det(\Lambda)$ is defined as the area of a fundamental domain of $\Lambda$. This can be seen by evaluating the integral in (\ref{eq:quadratureWeightFormula}) via periodization with respect to $\Lambda$.
\end{remark}

\section{Discretization of the Short-Time Fourier Transform}

Short-time Fourier transforms (STFT)~\cite{Groechenig2001} and their discretizations, known as Gabor transforms, are central tools of time-frequency signal processing. Furthermore, STFTs are likely the best understood family of continuous frames, with plenty of related discretization results. It seems worthwhile to investigate what can be achieved using QMC in this well-studied setting, though this is probably not the most radical application of our approach.

Given a window function $g\in \LL^2(\RR)$, we define the STFT by $V_g: \LL^2(\RR)\to \LL^2(\RR^2)$,
\[ V_gf(x,\omega) = \int_\RR f(t) \overline{g(t-x)} e^{-2\pi i t \omega} \,dt. \]
Using time-frequency shifts $\pi(x,\omega)g(t) = g(t-x)e^{2\pi i x\omega}$ for $x,\omega\in\RR$, the STFT can be written as $V_gf(x,\omega) = \langle f, \pi(x,\omega) g\rangle$. We assume $\norm{g}_{\LL^2} = 1$, as $V_g$ then becomes a unitary mapping. This is equivalent to $\mathcal{G} = (\pi(\PhVar) g)_{\PhVar\in\RR^2}$ being a continuous Parseval frame for $\LL^2(\RR)$.

Let $\PhVar,\PhVarA,\PhVarB\in\RR^2$. The reproducing kernel of $\mathcal{G}$ is given by
\[ R(\PhVar, \PhVarA) = \langle \pi(\PhVarA) g, \pi(\PhVar) g \rangle = V_g(\pi(\PhVarA)g)(\PhVar). \]
The kernel $K^{(\PhVar, \PhVarA)}(\PhVarB) = R(\PhVar, \PhVarB)R(\PhVarB, \PhVarA)$ is therefore equal to
\begin{align*}
K^{(\PhVar, \PhVarA)}(\PhVarB) 
&= \overline{V_g (\pi(\PhVar)g)(\PhVarB)} V_g (\pi(\PhVarA)g)(\PhVarB).
\end{align*}
Using the notation $Dg(t) = g'(t)$ and $Zg(t) = 2\pi i t\, g(t)$, some straightforward computations allow to estimate expression (\ref{eq:oscillationOfFrame}) for the STFT by
\begin{multline*} \Omega(\mathcal{G}) \leq 2\Big[ \norm{V_gg}_{\LL^1} \norm{V_gDg}_{\LL^1} + \norm{V_gg}_{\LL^1} \norm{V_gZg}_{\LL^1} \\
\quad+ \norm{V_gg}_{\LL^1} \norm{V_gZDg}_{\LL^1} + \norm{V_gDg}_{\LL^1} \norm{V_gZg}_{\LL^1}\Big].
\end{multline*}
This inequality enables us to apply Theorem \ref{thm:sufficientConditionDiscretization} to the STFT to obtain Gabor frames of $\LL^2(\RR)$, provided that the set $\Lambda$ has sufficiently small discrepancy $\discTrans(\Lambda)$.

\begin{example}
For $\sigma>0$, let $g_\sigma(t) = \sigma^{-1/2} \exp( -\frac{\pi}{2\sigma^2} t^2)$ be the $\LL^2$-normalized Gaussian and $\mathcal{G}_\sigma$ the corresponding continuous frame. Then $\Omega(\mathcal{G}_\sigma)$ can be evaluated exactly as
\[ \Omega(\mathcal{G}_\sigma) = 4\pi \left( \frac{1}{\sqrt{2}\sigma} + \sqrt{2}\sigma \right) + 4\pi^2. \]
It is minimal for $\sigma_0 = \sqrt{2}^{-1}$ with $\Omega(\mathcal{G}_{\sigma_0}) \approx 64.61$.
\end{example}

\section{Discrepancy of Admissible Lattices}\label{sec:admissibleLattices}

There are certain lattices $\Gamma\subset\RR^2$ which have excellent discrepancy, and which are therefore well suited for our purposes. We use the same terminology as \cite{Skriganov1994}.

\begin{definition}\label{def:admissibleLattice}
Let $\Gamma\subset\RR^2$ be a lattice. We call $\Gamma$ \emph{admissible} if
\[ \inf_{(\gamma_1, \gamma_2)\in\Gamma\backslash \{0\} } |\gamma_1 \gamma_2| > 0. \]
\end{definition}

Admissible lattices are considered frequently in QMC integration, see for instance \cite{Ullrich2016,Kacwin2021} and the references therein.

The following proposition is an immediate consequence of \cite[Cor. 2.1]{Skriganov1994}. We write $\Gamma_\tau$, $\tau>0$ for the dilated lattice
\[ \Gamma_\tau = \{ (\tau \gamma_1, \tau^{-1}\gamma_2) \;|\; (\gamma_1, \gamma_2)\in\Gamma \}. \]

\begin{proposition}\label{prop:discrepancyAdmissibleLattice}
Let $\Gamma\subset\RR^2$ be an admissible lattice. Then there is a constant $C=C(\Gamma)$ such that
\begin{equation}\label{eq:admissibleLatticeAsymptotic}
\discTrans(a\Gamma_\tau) \leq C a^2 \ln ( 2+a^{-1} )
\end{equation}
for all $a\in (0,1)$ and all $\tau>0$.
\end{proposition}

If $\Gamma$ is an admissible lattice, the consequences of Proposition \ref{prop:discrepancyAdmissibleLattice} are twofold.

First, if we ignore the dilation by $\tau$ for a moment, we see that $\discTrans(a\Gamma)$ has a fast asymptotic decay for $a\to 0$. For comparison: the discrepancy of the scaled integer lattice $\discTrans(a\ZZ^2)$ only has an asymptotic decay linear in $a$. In two dimensions, the rate of decay (\ref{eq:admissibleLatticeAsymptotic}) is in fact best possible \cite{Roth1954}.

Second, if we allow arbitrary $\tau>0$ again, we see that this decay is \emph{uniform for all dilations of $\Gamma$}. That is, if we write
\[ \discDil(\Gamma) = \sup_{\tau>0} \discTrans(\Gamma_\tau), \]
then $\discDil(a\Gamma)$ also has the asymptotic decay (\ref{eq:admissibleLatticeAsymptotic}).

We can utilize the dilation invariance in the following way: For $h\in \LL^1(\RR^2)$ with continuous mixed partial derivatives, we define the dilation of $h$ by $h_\tau(\PhVar) = h(\tau\PhVar_1, \tau^{-1}\PhVar_2)$ where $\PhVar=(\PhVar_1,\PhVar_2)\in\RR^2$. It is easy to see that $e(h_\tau, a\Gamma) = e(h, a\Gamma_\tau)$ (remember that, according to remark \ref{rem:weightsForLattices}, the weights corresponding to $\Gamma$ and $\Gamma_\tau$ are all equal to $\det(\Gamma)=\det(\Gamma_\tau)$ ). Due to Proposition \ref{prop:globalKoksmaHlawka} we have the estimate
\begin{align}
|e(h_\tau, a\Gamma)| &\leq \discTrans(a\Gamma_\tau) \left[ \,\norm{\partial_1 h}_{\LL^1} + \norm{\partial_2 h}_{\LL^1} + \norm{\partial_{12} h}_{\LL^1}\right] \nonumber\\
&\leq \discDil(a\Gamma) \left[ \,\norm{\partial_1 h}_{\LL^1} + \norm{\partial_2 h}_{\LL^1} + \norm{\partial_{12} h}_{\LL^1}\right].  \label{eq:quadratureWithDilation}
\end{align}
Thus, for $a\to 0$, the error $e(h_\tau, a\Gamma)$ decays fast and uniformly in $\tau$.

Let us apply this to the STFT. If we replace the window function $g\in\LL^2(\RR)$ by its $\LL^2$-normalized dilation $g_\tau(t) = \tau^{-1/2} g(\tau^{-1} t)$, the kernel $K^{(\PhVar,\PhVarA)}(\PhVarB)$ becomes
\[ K^{\left( (\tau^{-1}\PhVar_1, \tau \PhVar_2), (\tau^{-1}\PhVarA_1, \tau\PhVarA_2) \right)} ( \tau^{-1}\PhVarB_1, \tau\PhVarB_2). \]
The Schur norm in Proposition \ref{prop:discretizationSufficientCondition} is not affected by the dilations of $\PhVar$ and $\PhVarA$. Thus, combining Proposition \ref{prop:discretizationSufficientCondition} with inequality (\ref{eq:quadratureWithDilation}) yields the following variation of Theorem \ref{thm:sufficientConditionDiscretization}.

\begin{proposition}
Let $\Gamma\subset\RR^2$ be an admissible lattice and assume that $\discDil(\Gamma)\Omega(g) < 1$. Then, for all $\tau>0$, the family
\[ (\sqrt{\det(\Gamma)} \pi(\gamma)g_\tau)_{\gamma\in\Gamma} \]
is a Gabor frame with frame bounds $1-\discDil(\Gamma)\Omega(g)$ and $1+\discDil(\Gamma)\Omega(g)$ which are uniform in $\tau$.
\end{proposition}

\begin{example}
If we denote the golden ratio by $\varphi$, the lattice
\[ \Gamma = \left(\begin{matrix} 1 & \varphi^{-1} \\ -\varphi^{-1} & 1 \end{matrix}\right)\ZZ^2 \]
is admissible. In \cite{Holighaus2023}, the authors used this lattice (and implicitly its admissibility) to discretize the wavelet transform through the oscillation method.

In more generality, given $r,s,u,v\in\RR\backslash\{0\}$, the lattice
\[ \Gamma = \left(\begin{matrix} r & s \\ u & v \end{matrix}\right)\ZZ^2 \]
is admissible if and only if $r/s$ and $u/v$ are distinct irrational \emph{badly-approximable numbers} \cite[Ex. 2.10]{Matousek2010}\footnote{The exercise erroneously states that the quotients $r/u$ and $s/v$ have to be considered for the equivalence.} (see \cite[App. D]{Bugeaud2012} for a definition of badly-approximable numbers).
\end{example}

\begin{remark}
The generalization of our theory to $\RR^d$ for arbitrary $d\in\NN$ is straightforward. For example, if we assume $h\in L^1(\RR^d)$ to have continuous mixed partial derivatives, and if the definition of $\discTrans(\Lambda)$ is extended to $\Lambda\subset\RR^d$ in the obvious way, Proposition \ref{prop:globalKoksmaHlawka} becomes
\[ |e(h,\Lambda)| \leq \discTrans(\Lambda) \sum_{\emptyset\neq u\subset \{1,\dots, d\}} \norm{\partial_u h}_{L^1}. \]
\end{remark}

\bibliographystyle{IEEEtran}
\bibliography{bibliography}

\begin{thebibliography}{10}
\providecommand{\url}[1]{#1}
\csname url@samestyle\endcsname
\providecommand{\newblock}{\relax}
\providecommand{\bibinfo}[2]{#2}
\providecommand{\BIBentrySTDinterwordspacing}{\spaceskip=0pt\relax}
\providecommand{\BIBentryALTinterwordstretchfactor}{4}
\providecommand{\BIBentryALTinterwordspacing}{\spaceskip=\fontdimen2\font plus
\BIBentryALTinterwordstretchfactor\fontdimen3\font minus
  \fontdimen4\font\relax}
\providecommand{\BIBforeignlanguage}[2]{{%
\expandafter\ifx\csname l@#1\endcsname\relax
\typeout{** WARNING: IEEEtran.bst: No hyphenation pattern has been}%
\typeout{** loaded for the language `#1'. Using the pattern for}%
\typeout{** the default language instead.}%
\else
\language=\csname l@#1\endcsname
\fi
#2}}
\providecommand{\BIBdecl}{\relax}
\BIBdecl

\bibitem{Freeman2019}
\BIBentryALTinterwordspacing
D.~Freeman and D.~Speegle, ``The discretization problem for continuous
  frames,'' \emph{Adv. Math.}, vol. 345, pp. 784--813, 2019. URL:
  \url{https://doi.org/10.1016/j.aim.2019.01.006}
\BIBentrySTDinterwordspacing

\bibitem{Fuehr2017}
\BIBentryALTinterwordspacing
H.~F{\"u}hr, K.~Gr{\"o}chenig, A.~Haimi, A.~Klotz, and J.~L. Romero, ``Density
  of sampling and interpolation in reproducing kernel {H}ilbert spaces,''
  \emph{J. Lond. Math. Soc. (2)}, vol.~96, no.~3, pp. 663--686, 2017. URL:
  \url{https://doi.org/10.1112/jlms.12083}
\BIBentrySTDinterwordspacing

\bibitem{Dick2010}
\BIBentryALTinterwordspacing
J.~Dick and F.~Pillichshammer, \emph{Digital nets and sequences}.\hskip 1em
  plus 0.5em minus 0.4em\relax Cambridge University Press, Cambridge, 2010,
  discrepancy theory and quasi-Monte Carlo integration. URL:
  \url{https://doi.org/10.1017/CBO9780511761188}
\BIBentrySTDinterwordspacing

\bibitem{Skriganov1994}
M.~M. Skriganov, ``Constructions of uniform distributions in terms of geometry
  of numbers,'' \emph{Algebra i Analiz}, vol.~6, no.~3, pp. 200--230, 1994.

\bibitem{fornasier2005}
\BIBentryALTinterwordspacing
M.~Fornasier and H.~Rauhut, ``Continuous frames, function spaces, and the
  discretization problem,'' \emph{J. Fourier Anal. Appl.}, vol.~11, no.~3, pp.
  245--287, 2005. URL: \url{https://doi.org/10.1007/s00041-005-4053-6}
\BIBentrySTDinterwordspacing

\bibitem{Groechenig1991}
\BIBentryALTinterwordspacing
K.~Gr\"ochenig, ``Describing functions: atomic decompositions versus frames,''
  \emph{Monatsh. Math.}, vol. 112, no.~1, pp. 1--42, 1991. URL:
  \url{https://doi.org/10.1007/BF01321715}
\BIBentrySTDinterwordspacing

\bibitem{Dahlke2008}
\BIBentryALTinterwordspacing
S.~Dahlke, M.~Fornasier, H.~Rauhut, G.~Steidl, and G.~Teschke, ``Generalized
  coorbit theory, {B}anach frames, and the relation to {$\alpha$}-modulation
  spaces,'' \emph{Proc. Lond. Math. Soc. (3)}, vol.~96, no.~2, pp. 464--506,
  2008. URL: \url{https://doi.org/10.1112/plms/pdm051}
\BIBentrySTDinterwordspacing

\bibitem{Christensen2016}
\BIBentryALTinterwordspacing
O.~Christensen, \emph{An introduction to frames and {R}iesz bases}, 2nd~ed.,
  ser. Applied and Numerical Harmonic Analysis.\hskip 1em plus 0.5em minus
  0.4em\relax Birkh\"auser/Springer, [Cham], 2016. URL:
  \url{https://doi.org/10.1007/978-3-319-25613-9}
\BIBentrySTDinterwordspacing

\bibitem{Groechenig2001}
\BIBentryALTinterwordspacing
K.~Gr\"ochenig, \emph{Foundations of time-frequency analysis}, ser. Applied and
  Numerical Harmonic Analysis.\hskip 1em plus 0.5em minus 0.4em\relax
  Birkh\"auser Boston, Inc., Boston, MA, 2001. URL:
  \url{https://doi.org/10.1007/978-1-4612-0003-1}
\BIBentrySTDinterwordspacing

\bibitem{Ullrich2016}
\BIBentryALTinterwordspacing
M.~Ullrich and T.~Ullrich, ``The role of {F}rolov's cubature formula for
  functions with bounded mixed derivative,'' \emph{SIAM J. Numer. Anal.},
  vol.~54, no.~2, pp. 969--993, 2016. URL:
  \url{https://doi.org/10.1137/15M1014814}
\BIBentrySTDinterwordspacing

\bibitem{Kacwin2021}
\BIBentryALTinterwordspacing
C.~Kacwin, J.~Oettershagen, M.~Ullrich, and T.~Ullrich, ``Numerical performance
  of optimized {F}rolov lattices in tensor product reproducing kernel {S}obolev
  spaces,'' \emph{Found. Comput. Math.}, vol.~21, no.~3, pp. 849--889, 2021.
  URL: \url{https://doi.org/10.1007/s10208-020-09463-y}
\BIBentrySTDinterwordspacing

\bibitem{Roth1954}
\BIBentryALTinterwordspacing
K.~F. Roth, ``On irregularities of distribution,'' \emph{Mathematika}, vol.~1,
  pp. 73--79, 1954. URL: \url{https://doi.org/10.1112/S0025579300000541}
\BIBentrySTDinterwordspacing

\bibitem{Holighaus2023}
N.~Holighaus and G.~Koliander, ``Rotated time-frequency lattices are sets of
  stable sampling for continuous wavelet systems,'' in \emph{2023 International
  Conference on Sampling Theory and Applications (SampTA)}, 2023, pp. 1--5.

\bibitem{Matousek2010}
\BIBentryALTinterwordspacing
J.~Matou{\v s}ek, \emph{Geometric discrepancy}, ser. Algorithms and
  Combinatorics.\hskip 1em plus 0.5em minus 0.4em\relax Springer-Verlag,
  Berlin, 2010, vol.~18, an illustrated guide, Revised paperback reprint of the
  1999 original. URL: \url{https://doi.org/10.1007/978-3-642-03942-3}
\BIBentrySTDinterwordspacing

\bibitem{Bugeaud2012}
\BIBentryALTinterwordspacing
Y.~Bugeaud, \emph{Distribution modulo one and {D}iophantine approximation},
  ser. Cambridge Tracts in Mathematics.\hskip 1em plus 0.5em minus 0.4em\relax
  Cambridge University Press, Cambridge, 2012, vol. 193. URL:
  \url{https://doi.org/10.1017/CBO9781139017732}
\BIBentrySTDinterwordspacing

\end{thebibliography}

\end{document}